\documentclass{amsart}
\usepackage{amsthm,amsfonts,amsmath,amscd,amssymb,latexsym,epsfig,mathrsfs}
\usepackage{xcolor}
\usepackage[pdftex,pdfpagelabels,bookmarks,hyperindex,hyperfigures]{hyperref}
\hypersetup{
    colorlinks = true,
    linkcolor = {blue},
   linkbordercolor = {white},
   citecolor={magenta},
    linktocpage=true
}
\usepackage{enumitem}

\newcommand{\so}{{\mathfrak s \mathfrak o}}

\newcommand{\ssl}{{\mathfrak s \mathfrak l}}

\newcommand{\g}{{\mathfrak g}}         

\newcommand{\cx}{{\mathbb C}}

\newcommand{\Ad}{\operatorname{Ad}}
\newcommand{\tr}{\operatorname{tr}}

\newcommand{\codim}{\operatorname{codim}}

\newcommand{\End}{\operatorname{End}}

\newcommand{\Mat}{\operatorname{Mat}}

\newcommand{\Sing}{\operatorname{Sing}}
\newcommand{\Reg}{\operatorname{Reg}}

\newcommand{\ol}{\overline}

\numberwithin{equation}{section}

\newtheorem{theorem}{Theorem}[section]

\newtheorem{corollary}[theorem]{Corollary}

\newtheorem{proposition}[theorem]{Proposition}

\theoremstyle{remark}

\newtheorem{remark}[theorem]{Remark}

\newtheorem{definition}[theorem]{Definition}

\newtheorem*{ack}{Acknowledgment}

\newcommand{\C}{{\mathbb{C}}}

\newcommand{\R}{{\mathbb{R}}}

\newcommand{\oC}{{\mathbb{C}}}

\newcommand{\oH}{{\mathbb{H}}}

\newcommand{\oK}{{\mathbb{K}}}

\newcommand{\oP}{{\mathbb{P}}}

\newcommand{\oR}{{\mathbb{R}}}

\newcommand{\oZ}{{\mathbb{Z}}}

\newcommand{\sD}{{\mathcal{D}}}
\newcommand{\sE}{{\mathcal{E}}}

\newcommand{\sN}{{\mathcal{N}}}
\newcommand{\sO}{{\mathcal{O}}}

\makeatletter
\@namedef{subjclassname@2020}{\textup{2020} Mathematics Subject Classification}
\makeatother

\begin{document}

\title{Hypercomplex analytic spaces and schemes}
\author{Roger Bielawski }
\address{Institut f\"ur Differentialgeometrie,
Leibniz Universit\"at Hannover,
Welfengarten 1, 30167 Hannover, Germany}




\begin{abstract} We propose definitions of hypercomplex analytic spaces and hypercomplex schemes. We show that   such a hypercomplex space is canonically associated to the quotient of a hypercomplex manifold by a finite group action.
\end{abstract}

\maketitle

\thispagestyle{empty}


The geometry based on the largest of real division algebras, i.e.\ on the quaternions $\oH$, has a very different flavour, owing to the noncommutativity. In particular, it cannot be based on ``quaternionic coordinates", since a quaternion-analytic map $\oH^n\to \oH^n$ is necessarily affine. Instead, one defines {\em hypercomplex manifolds} as real manifolds admitting an action of quaternions on the tangent bundle, such that the almost complex structures corresponding to unit imaginary quaternions are integrable. This definition does not adapt well to the singular setting. On the other hand there are many natural examples of ``singular spaces" which should carry some sort of hypercomplex (or hyperk\"ahler) structure. Let us mention three classes of such examples:
\begin{itemize}
\item quotients of hypercomplex manifolds by a non-free triholomorphic action of a finite group;
\item singular hyperk\"ahler quotients (cf., e.g., \cite{DS, May});
\item Coulomb branches in $3$-dimensional $N=4$ supersymmetric gauge theories (see, e.g., \cite{Nak}).
\end{itemize}
In each of these cases,  the  singular object has a natural structure of a {\em stratified} hypercomplex or hyperk\"ahler manifold. This, however, is a very weak notion, rather far from algebraic or analytic varieties. The 
aim of this paper is to propose  a definition of ``singular hypercomplex" in both the (real) analytic and the algebraic category. We shall concentrate mostly on {\em hypercomplex analytic spaces}, because the algebraic notion is too restrictive; for example, it does not include  the hypercomplex structure of the Taub-NUT metric on $\R^4$, or the flat hypercomplex structure on $S^1\times \oR^3$. The latter example, together with the description of an open dense subset of the smooth locus of a Coulomb branch given in   \cite{BDG},  also shows that  the hypercomplex structure of a Coulomb branch is never algebraic. 
\par
Let us briefly recall the previous work on this topic. A notion of a hypercomplex variety has been proposed by Verbitsky \cite{Ver}. It is a direct generalisation of the definition of a hypercomplex manifold, requiring an action of quaternions on the cotangent sheaf such that the unit imaginary quaternions define integrable complex structures. As shown by Verbitsky, this definition is very restrictive: hypercomplex singularities in this sense can always be resolved by normalisation. In particular the quotient of a hypercomplex manifold by a non-free action of a finite group is never a hypercomplex variety in this sense. 
\par
There is also Joyce's {\em hypercomplex algebraic geometry} \cite{Joyce}, which associates certain kind of quaternionic algebras (called H-algebras) to  hypercomplex manifolds. However, as Joyce himself points out \cite[p.\ 140]{Joyce}, the geometry of the variety associated to an arbitrary H-algebra may not be hypercomplex.
\par
So what should a singular hypercomplex space or scheme be? 
We clearly should have the usual hypercomplex structure on the set $\Reg(X)=X\setminus\Sing(X)$ of regular points. This implies that the sheaf $\varOmega^{[1]}_X=i_\ast \varOmega^1_{\Reg(X)}$  is quaternionic (where $i:\Reg(X)\to X$ is the embedding), i.e.\ $\End_{\sO_X}\bigl({\varOmega}^{[1]}_X\bigr)$ contains a subalgebra isomorphic to quaternions. In addition, we need a notion of integrability for almost complex structures on $\varOmega^{[1]}_X$.
 Consider first a real analytic manifold $M$ with an almost complex structure $I:TM\to TM$. The integrability of $I$ can be interpreted as follows: there is a complex thickening $\tilde M$ of $M$ such that the distribution given by the  $-i$-eigenbundle $T^{-}\tilde M\subset T\tilde M$  of $I$ is integrable, i.e.\ $T_x^{-}\tilde M$ is the tangent space to a complex submanifold (of dimension $\frac{1}{2}\dim \tilde M$) of $\tilde M$ at every $x\in \tilde M$.  Let now $X$ be a pure-dimensional real analytic space (or algebraic scheme) with an almost complex structure $I$ on the sheaf $\varOmega^{[1]}_X$. Let $\tilde X$ be a complexification of $X$ and, assuming that $\varOmega^{[1]}_X$ is coherent,  consider the linear complex space (resp.\ linear scheme) $L^{-}$ over $\tilde X$, corresponding to the sheaf $\varOmega^{[-]}_{\tilde X}=i_\ast\varOmega^{-}_{\Reg(\tilde X)}$. It is a linear  subspace of the  tangent linear space (resp.\ scheme) $T\tilde X$, and we require, analogously to the smooth case, that at every $x\in \tilde X$, the fibre $L^{-}_x$ is equal to the Zariski tangent space of a complex subspace (resp.\ subscheme) of dimension $\frac{1}{2}\dim\tilde X$.
 Finally, we also require  that the intersection of fibres $L_x^{-}\subset T_x\tilde X$, corresponding to two different unit imaginary quaternions, remains trivial at singular points of $X$ (although spaces not satisfying this condition are also interesting; we shall call them {\em weakly hypercomplex}).
 \par
 This is then, in essence, our definition of a {\em normal}\hspace{1pt} hypercomplex space or scheme: a linear quaternionic structure on $\varOmega^{[1]}_X$, satisfying the above nondegeneracy condition at points of $\Sing(X)$, and such that the almost complex structures corresponding to unit imaginary quaternions are integrable in the above sense.
 We shall, however, formulate the definition in different terms, which do not require normality.
 \par
 We remark that this notion of integrability does not necessarily make $X$ into a complex space (or scheme). In the smooth case, one can guarantee, by making $\tilde M$ smaller if necessary, that the foliation of $\tilde M$ induced by $T^{-}\tilde M$ is simple, and therefore the space of leaves is diffeomorphic to $M$. This is, in general, not possible in the singular case: if $X$ is locally irreducible and  $\Sing(X)$ disconnects $X$ locally, then we cannot make $X$ into a complex space (with the same underlying real analytic structure). It seems that, in such a case, the best we can hope for, is the existence of a complex space $Z_I$ and a surjective morphism $\Phi_I$ from $X$ to the underlying real analytic space $Z_I^\R$ of $Z_I$, such that $\Phi_I$ has discrete fibres and is unbranched away from $\Sing(X)$. Indeed, it turns out that our definition of hypercomplex space guarantees the existence of such $Z_I$ for each unit imaginary quaternion, provided the complexification $\tilde X$ of $X$ is normal. We obtain then a ``twistor space" $\pi:Z\to \oP^1$ for $X$, equipped with an antiholomorphic involution $\sigma$ covering the antipodal map. There is a natural map $\Phi:X\times S^2\to Z^\R$ with discrete fibres, which identifies $X$ with an analytic subspace of  the Douady space $\Gamma(Z)^\sigma$ of $\sigma$-invariant sections of $\pi$. 
 \par
 Conversely, if $X$ is a real analytic subspace of $\Gamma(Z)^\sigma$ such that $\Reg(X)$ is open in $\Reg\bigl(\Gamma(Z)^\sigma\bigr)$ and the natural maps from $X$ to the fibres of $Z$, given by the intersection of a section with a fibre, are surjective with discrete fibres and unbranched away from $\Sing(X)$, then $X$ is a hypercomplex space.
 \par
 On the other hand, if $\Sing(X)$ does not disconnect $X$ locally, then we can expect $X$ to be canonically a complex space for each unit imaginary quaternion. We show that this is the case provided the complexification $\tilde X$ is normal and the map $\Phi$ is finite. This latter assumption is worth making even in the case when $\Sing(X)$  locally disconnects $X$. Hypercomplex spaces with finite $\Phi$ are, in a way,  the closest relatives of hypercomplex manifolds, and we shall call them {\em proper}. For example, a proper hypercomplex space $X$ with a twistor space $Z$ is isomorphic to a union of irreducible components of $\Gamma(Z)^\sigma$.
\par
We then discuss examples. We show that we can associate a canonical hypercomplex space  to a  finite quotient of a hypercomplex manifold. We cannot expect  the  quotient itself to be always a hypercomplex space or variety, for the simple reason that it does not even have to be  a real analytic space/variety. Consider the simplest example of 
$\oH/{\pm 1}$. If we identify $\oH$ with $\oR^4$ with coordinates $x_0,x_1,x_2,x_3$, then the invariants are the monomials $x_ix_j$, and $\oR^4/{\pm 1}$ can be identified with the {\em semialgebraic} variety of real symmetric $4\times 4$ matrices of rank $\leq 1$ and nonnegative trace. Its Zariski closure is the variety of all real symmetric $4\times 4$ matrices of rank $\leq 1$, which is the union of two copies of $\oH/{\pm 1}$ intersecting in a point, and this is the variety which we can reasonably expect to be hypercomplex. In fact, this variety is an example in Joyce's  hypercomplex algebraic geometry \cite[p.\ 161]{Joyce}. 
\par
Generalising this example, we show that there is a canonical proper hypercomplex space $\ol{M/G}$ associated  to a hypercomplex manifold $M$ with a  triholomorphic action of a finite group $G$. It is given as the real part of $M^\C/G$, where $M^\C$ is the connected component containing $M$ of the Douady space of sections of the twistor space of $M$.  The quotient $M/G$ is a closed subset (not necessarily a subspace) of $\ol{M/G}$. If the order of every element of $G$ is odd, then $\ol{M/G}=M/G$, so that $M/G$ is a hypercomplex space.
\par
A different generalisation of the  example $\oH/{\pm 1}$, also yielding  proper hypercomplex spaces, is given by conical complex spaces such that the underlying real analytic space has a compatible stratification into hypercomplex manifolds. We show that to such a complex space $Y$ we can associate a hypercomplex space $X$, homeomorphic to the union of two copies of $Y$ joined at the common cone vertex. In particular, the union of two copies of the nilpotent variety of a complex semisimple Lie algebra, glued together at the origin, is a proper integral hypercomplex scheme.

\begin{ack} I thank Daniel Greb for comments and reference \cite{HHL}.
\end{ack}

\section{Background material}
\subsection{Analytic spaces\label{ana-spaces}}

General references for the material in this subsection are \cite{VII} for complex and \cite{GMT} for real analytic spaces.
\par
Let $\oK=\oR$ or $\oC$. A {\em $\oK$-model space} is a ringed space $(X,\sO_X)$ where $X$ is the common zero-set of finitely many $\oK$-analytic functions $f_1,\dots,f_n$ defined on an open subset $U$ of $\oK^N$, and $\sO_X=\sO_U/(f_1\sO_U+\dots +f_n\sO_U)\big|_X$ ($\sO_U$ denotes the sheaf of germs of analytic functions on $U$). A $\oK$-analytic space is a second countable Hausdorff ringed space $(X,\sO_X)$ which is locally isomorphic to model spaces.  It follows from the definition that any real analytic space has a {\em complexification}, i.e.\ complex analytic space $(\tilde X, \sO_{\tilde X})$ equipped with an antiholomorphic involution $\sigma$ such that  $(X,\sO_X)=(\tilde X^\sigma, \sO_{\tilde X}^\sigma)$, where $\sO_{\tilde X}^\sigma=\{s\in \sO_{\tilde X}; s=\ol{s\circ\sigma}\}$.
\begin{remark} To a $\oK$-analytic space one associates its {\em reduction} $(X,\sO_X^{\rm red})$, where the sheaf $\sO_X^{\rm red}$ is obtained by replacing the sheaf $\sO_{X\cap V}=\sO_U/(f_1\sO_U+\dots +f_n\sO_U)$ of a model space with sheaf $\sO_U/I_{X\cap V}$, where the ideal sheaf $I_{X\cap V}$ consists of all $\oK$-analytic functions vanishing on $X\cap V$. In the complex case, the ringed space $(X,\sO_X^{\rm red})$ is still an analytic space, but it may fail to be so in the real case (the ideal sheaf $I_{X\cap V}$ does not have to be locally finitely generated).
\end{remark}
A complex analytic space $(X,\sO_X)$ has an underlying real analytic space $(X,\sO_{X}^\R)$ \cite[Def. II.4.1]{GMT}, the local models of which are given by the real and the imaginary parts  of the holomorphic functions defining the local models of $(X,\sO_X)$.
A subset $Y$ of an analytic space  $(X,\sO_X)$ is called {\em analytic} if every point $y\in Y$ has a neighbourhood $V\subset X$ such that $V\cap Y$ is the zero set of finitely many functions in $\sO_X(V)$. If $I$ is the ideal generated by these function, then setting  $\sO_Y(V\cap Y)=\bigl(\sO_X(V)/I\bigr)\big|_{Y}$ makes $(Y,\sO_Y)$ into an analytic space, called a {\em subspace of $X$}.
\par
Let $X$ be a real analytic space with a complexification $\tilde X$.
The singular locus of $X$ is $\Sing(X)=\Sing(\tilde X)\cap X$. It does not depend on the choice of a complexification and it is a closed subspace of $X$. The complement of $\Sing(X)$ will be denoted by $\Reg(X)$.
 $\Sing(X)$ may contain smooth points, i.e.\ points which have a neighbourhood analytically diffeomorphic to an open subset of a Euclidean space.
\par
A (real or complex) analytic space  is said to be {\em irreducible} if it is not the union of two nonempty closed  subspaces. Any analytic space $X$ is the union of  countably many irreducible analytic subspaces, called the {\em irreducible components} of $X$. If $X$ is reduced and $\oK=\C$, then the irreducible components are closures of the connected components of $\Reg(X)$, but this is false if $\oK=\oR$.
\par
The (analytic) dimension $\dim_x X$ of a complex space $X$ at a point $x\in X$ is the smallest integer $d$ such that $d$ holomorphic functions have $x$ as an isolated zero. It is equal to the topological dimension. The (real) analytic dimension $\dim_x X$ of a real analytic space is equal to the (complex) analytic dimension $\dim_x\tilde X$ of any complexification of $X$ (and, in general, is not equal to the topological dimension). A $\oK$-analytic space is said to be {\em pure-dimensional} if $x\mapsto\dim_x X$ is a constant function.
\par
A  nowhere dense closed subspace $T$ of a $\oK$-analytic space $X$ is called {\em thin}. A finite\footnote{A continuous map is called {\em finite} if it is finite-to-one and closed.} surjective analytic map $\phi:X\to Y$ between $\oK$-analytic spaces is called an {\em analytic covering} (of $Y$) if there exists a thin subspace $T\subset Y$ such that $\phi^{-1}(T)$ is thin in $X$, and the induced map $X\setminus \phi^{-1}(T)\to Y\setminus T$ is a local isomorphism (i.e.\ a local real analytic diffeomorphism if $\oK=\R$, and a local biholomorphism if $\oK=\oC$). The set $B\subset \phi^{-1}(T)$, where $\phi$ is not a local isomorphism, is called the {\em branch locus} of $\phi$. We remark that, if $\oK=\C$ and $Y$ is reduced, then any surjective, open, and finite mapping $\phi:X\to Y$ is an analytic covering \cite[Thm. I.12.11]{VII}.
\par
An equivalence relation $R\subset X\times X$ on a reduced complex space $(X,\sO_X)$ is called {\em analytic} if $R$ is a closed complex subspace of $ X\times X$.   We obtain a ringed space $(X/R,\sO_X^R)$, where $X/R$ is set of equivalence classes equipped with the quotient topology, and $\sO_X^R$ is the sheaf of holomorphic functions on $X$ which are  constant on fibres of $X\to X/R$.  In general,  $(X/R,\sO_X^R)$ is not a complex space (cf.\ \cite[Ch. IV]{VII}).

\subsection{Hypercomplex manifolds\label{hc-mflds}}

A (real) manifold $M$ is hypercomplex if its tangent bundle $TM$ admits a fibrewise action of quaternions such that the almost complex structure corresponding to any unit imaginary quaternion  is integrable. An important property of a hypercomplex structure is the existence of a unique torsion-free connection, called the {\em Obata connection}, for which all of these complex structures are parallel. To a hypercomplex manifold we can associate its {\em twistor space}, i.e.\ the manifold $M\times S^2$ with an almost complex structure $J_q\oplus J_{\oP^1}$ at the points of $M\times \{q\}$, where $J_q$ denotes the complex structure on $M$ given by the unit imaginary quaternion $q\in S^2$, and $J_{\oP^1}$ is the standard complex structure of $S^2\simeq \oP^1$. It has been shown by Simon Salamon \cite{Sal} that this complex structure is integrable and the projection $\pi:M\times S^2\to \oP^1$ onto the second factor is holomorphic. The antipodal map on $S^2$ induces an antiholomorphic involution $\sigma$ on $Z=M\times S^2$, and points of $M$ yield $\sigma$-equivariant sections of $\pi$ with normal bundle isomorphic to $\sO_{\oP^1}(1)^{\oplus 2n}$, where $\dim_\R M=4n$.
\par
Conversely, given a complex manifold $Z$  with a holomorphic submersion $\pi:Z\to \oP^1$ and an antiholomorphic involution $\sigma$ covering the antipodal map, the family of $\sigma$-equivariant sections of $\pi$ with normal bundle $\sO_{\oP^1}(1)^{\oplus 2n}$ (where $\dim_\C Z=2n+1$) is naturally a (possibly disconnected) hypercomplex manifold.
The existence of the manifold structure on the space of sections of the twistor space is the consequence of a much more general theory. Douady \cite{Dou} has shown that if $Z$ is a complex space, then the family of all compact complex subspaces of $Z$ is a complex space $\sD(Z)$. Moreover, the Zariski tangent space at each $C\in \sD(Z)$ is naturally isomorphic to $H^0(C,\sN_{C/Z})$, where $\sN_{C/Z}$ is the normal sheaf of $C$ in $Z$. Furthermore, if $C$ is a local complete intersection in $Z$ and $H^1(C,\sN_{C/Z})=0$, then $\sD(Z)$ is smooth at $C$. In the special case of compact complex submanifolds of a complex manifold, these results have been proved earlier by Kodaira \cite{Kod}.
\par
In the case of $Z$ equipped with a surjective holomorphic map $\pi:Z\to \oP^1$, we can consider the open and closed subspace $\Gamma(Z)$ of $\sD(Z)$ consisting of compact subspaces such that $\pi|_C:C\to \oP^1$ is an isomorphism, which is precisely the space of sections of $\pi$. In the case when $Z$ is the twistor space of a connected hypercomplex manifold $M$, as defined above, the space of sections corresponding to points of $M$ is not only a complete and maximal family as defined in \cite{Kod}, but even a connected component of the Douady space $\Gamma(Z)^\sigma$ of all $\sigma$-equivariant sections. Indeed, the above results of Kodaira imply that $M$ is open in $\Gamma(Z)^\sigma$. For a $\zeta\in \oP^1$, the morphism $\Phi_\zeta:\Gamma(Z)^\sigma\to Z_\zeta=\pi^{-1}(\zeta)$, given by the intersection of a section with fibre,  defines a retraction $r:  \Gamma(Z)^\sigma\to M$. Since the image of a retraction is closed, $M$ is connected (by assumption), open, and closed in $\Gamma(Z)^\sigma$.
\par
It has been shown in \cite{B-sigma} that one obtains a hypercomplex structure on more general subspaces of a Douady space. It is namely enough that $Z$ is a pure-dimensional complex space with a holomorphic surjection $\pi:Z\to \oP^1$ and  an antiholomorphic involution $\sigma$ covering the antipodal map. Then any connected component of the subset  of the Douady space $\sD(Z)^\sigma$ consisting of $\sigma$-invariant curves $C$ such that
\begin{enumerate}[label=(\roman*)]
\item $\pi:C\to \oP^1$ is flat;
\item $C$ is a local complete intersection in $Z$;
\item $h^0(C,\sN_{C/Z}(-2))=h^1(C,\sN_{C/Z}(-2))=0$,
\end{enumerate}
is a hypercomplex manifold.

\subsection{$\C$-hypercomplex manifolds\label{chc-mflds}}
Complexified quaternions have been named {\em biquaternions} by W.R.\ Hamilton. Their algebra is isomorphic to the algebra $\Mat_{2,2}(\cx)$ of $2\times 2$ complex matrices. We can consider complex manifolds, the tangent space of which admits an action of this algebra (with additional properties). The name {\em biquaternionic manifold} should be, however, properly reserved for the complex analogue of quaternionic manifolds, and therefore we shall call a complex version of hypercomplex simply $\C$-hypercomplex.
\begin{definition}  A complex manifold $M$ is called {\em almost  $\C$-hypercomplex} if its holomorphic tangent bundle $T M$ decomposes as $T M \simeq
E \otimes \C^2$, where $E$ is a holomorphic vector bundle. It is {\em $\C$-hypercomplex} if, in addition, for any $v \in \C^2$ the
subbundle $E\otimes v$ defines an integrable distribution on $M$.\label{E-bundle}\end{definition}
\begin{remark} Almost $\C$-hypercomplex is equivalent to having a linear action of $\Mat_{2,2}(\cx)$ on $TM$. If $M$ is (almost) $\C$-hypercomplex, then so is $M$ with the opposite complex structure.  A holomorphic or an antiholomorphic map between $\C$-hypercomplex manifolds is said to be {\em $\C$-hypercomplex} if its differential commutes with the respective actions of biquaternions.
\end{remark}
$\C$-hypercomplex manifolds arise naturally via the twistor construction. If $\pi:Z^{2n+1}\to \oP^1$ is a surjective holomorphic map from a complex manifold, then the open subspace of the Douady space $\sD(Z)$ consisting of sections with normal bundle isomorphic to $\sO_{\oP^1}(1)^{\oplus 2n}$
 is a $\C$-hypercomplex manifold. In particular, if $Z$ is the twistor space of a hypercomplex manifold $X$, then we obtain this way a natural complexification of $X$ and its hypercomplex structure. Conversely, 
given a $\cx$-hypercomplex manifold $M$ we can define an integrable distribution $\sE$ on $M\times \oP^1$ by $\sE_{|M\times\{z\}}=E\otimes h$, where $h$ is the highest weight vector for the maximal torus in $SL(2,\cx)$ corresponding to $z\in \oP^1$. If the leaf space $Z$ of $\sE$ is  a manifold, then $Z$ is the twistor space of $M$. Otherwise, one needs to view $Z$ in terms of foliated geometry.
\par
There is another way to view $\C$-hypercomplex  structures. Observe that, for any $v\neq w\in \oC^2$, the subbundles $E\otimes v$ and $E\otimes w$ of $TM$ are everywhere transversal. Therefore a $\C$-hypercomplex structure can be also viewed as a special family of local product structures  parametrised by $(\zeta,\eta)\in (\oP^1\times \oP^1)\setminus\Delta$. If the twistor space of $M$ is a manifold, then these structures are global ``up to a covering", in the sense that there exist complex manifolds $ M_\zeta$, $\zeta\in \oP^1$, and, for any $(\zeta,\eta)\in (\oP^1\times \oP^1)\setminus\Delta$,  a surjective local biholomorphism  $M\to M_\zeta\times M_\eta$.

\section{Hypercomplex spaces}

We begin by making  additional observations about the smooth case.
\par
Let $M$ be a $2n$-dimensional real-analytic manifold with an integrable almost complex structure $I\in \End(TM)$. We can extend $I$ to a complexification (cf.\ \S\ref{ana-spaces})  $(\tilde M,\sigma)$, so that $\sigma^\ast I=-I$.  The $\pm i$-eigenbundles of $I$ define integrable distributions $T^\pm \tilde M$ and, consequently, two analytic equivalence relations (cf.\ \S\ref{ana-spaces})  $R_\pm$ on $\tilde M$: two points are related if and only if they belong to the same leaf of the distribution. The relations $R_\pm\subset \tilde M\times \tilde M$ satisfy the following three conditions: (1) $\sigma(R_-)=R_+$; (2) the equivalence classes of $R_\pm $ are complex $n$-dimensional submanifolds of $\tilde M$; and (3)  equivalence classes of $R_+$ are transversal to  equivalence classes of $R_-$ at any intersection point. Property (i) implies that $R_+$ is determined by $R_-$ and $\sigma$.
\par
Similarly,  if $M$ is a real analytic $4n$-dimensional hypercomplex manifold, then, as explained in \S\ref{chc-mflds}, its complexification $\tilde M$ has a natural $\cx$-hypercomplex structure, which may be viewed as an analytic $\oP^1$-family\footnote{A family $\{R_y;y\in Y\}$ of equivalence relations on a complex space $X$, parametrised by a complex space $Y$, is called an {\em analytic $Y$-family}, if the equivalence relation $R$ on $X\times Y$, equal to $R_y$ on $X\times\{y\}$, is analytic.} $R_\zeta$, $\zeta\in\oP^1$, of equivalence relations. 
\par
We now generalise this interpretation of hypercomplex structures to singular setting: 
\begin{definition} A real analytic space $X$ of pure dimension $4n$, such that $\Sing(X)$ is thin, is said to be  {\em weakly hypercomplex } if there exists a  reduced complexification $(\tilde X,\sigma)$ of $X$ and
an analytic $\oP^1$-family $\{R_\zeta;\zeta\in \oP^1\}$ of equivalence relations on $\tilde X$, with the following properties:
\begin{enumerate}[label=(\roman*)] 
\item the restrictions $R^\prime_\zeta$ of $R_\zeta$, $\zeta\in \oP^1$, to $\Reg(X)$ are the equivalence relations induced by a hypercomplex structure on $\Reg(X)$, and $R_\zeta=\ol{R^\prime_\zeta}$, $\forall\zeta\in \oP^1$;
\item every  equivalence class  of $R_\zeta$, $\zeta\in \oP^1$, intersects $X$ and its dimension at such an intersection point is $2n$;
\item $R_\zeta\cap \bigl(\Reg(X)\times \Sing(X)\bigr)=\emptyset$, $\forall\zeta\in \oP^1$.
 \item $\bigcap_{\zeta\in \oP^1} R_\zeta$ is the diagonal in $\tilde X\times \tilde X$.
\end{enumerate}
$X$ is said to be {\em hypercomplex}, if, in addition:
\begin{itemize}
\item[(v)] for every $\zeta\neq \eta$,
 the equivalence classes of $R_\zeta\cap R_\eta$ are discrete.
 \end{itemize}\label{hc-space}
\end{definition}
\begin{remark} {\em Weakly hypercomplex} and {\em  hypercomplex schemes} are defined analogously  by replacing analytic spaces with schemes of finite type over $\R$ or $\C$, ``discrete" with ``quasifinite", and analytic equivalence relations  with algebraic ones. \label{hc-scheme} \end{remark}
Here is a couple of immediate consequences of the definition:
\begin{proposition} Let $X$ be a weakly hypercomplex space with $\tilde X$, and $R_\zeta$, $\zeta\in \oP^1$, be as in Definition \ref{hc-space}. Then $\sigma({R_\zeta})=R_{-1/\bar\zeta}$, $\forall\zeta\in \oP^1$. If $X$ is hypercomplex, then the intersection $L\cap X$  of any equivalence class $L$ of any $R_\zeta$ with $X$ is discrete.
\label{consequences}
\end{proposition}
\begin{proof} The first statement  follows from (i).  If $L$ is an equivalence class of any $R_\zeta$, then $\sigma({R_\zeta})=R_{-1/\bar\zeta}$ forces $L\cap X$ to be equal to $L^\prime \cap X$, where $L^\prime$ is an equivalence class of $R_{-1/\bar\zeta}$. Now (v) implies that $L\cap X$ is discrete.
\end{proof}
The {\em hypercomplex structure} of a hypercomplex space (resp.\ the weak hypercomplex structure) is, properly seen, the germ of complexifications satisfying the conditions in Definition \ref{hc-space}. Equivalently, we can add the condition:
\begin{itemize}
\item[(vi)] for every $\zeta\in\oP^1$, if $x,y\in X$ are $R_\zeta$-equivalent, then they are equivalent with respect to $R_\zeta\cap (V\times V)$ for  every $\sigma$-invariant open neighbourhood $V\subset \tilde X$ of $X$. 
\end{itemize}
We observe that condition (i) implies that the hypercomplex structure of a hypercomplex space is uniquely determined by the hypercomplex structure of $\Reg(X)$.
\par
Let $\tilde X$ be an arbitrary complexification of $X$ satisfying conditions (i)--(iv) of Definition \ref{hc-space}, and 
consider the analytic relation $R$  on $\tilde X\times\oP^1$, equal to $R_\zeta$ on $\tilde X\times \{\zeta\}$.
If the ringed space $Z(\tilde X)=\bigl((\tilde X\times \oP^1)/R,(\sO_{\tilde X\times \oP^1)})^R)$ is a reduced complex space, then it is called a {\em twistor space for $X$}. It is clear that $Z(\tilde X)$ comes equipped with a surjective holomorphic map $\pi:Z(\tilde X)\to \oP^1$ and an (analytic) antiholomorphic involution $\sigma:Z(\tilde X)\to Z(\tilde X)$ covering the antipodal map. If $V\subset \tilde X$ is a $\sigma$-invariant open neighbourhood  of $X$, then $Z(V)$ is also a twistor space with a canonical surjective holomorphic map $Z(V)\to Z(\tilde X)$. If $\tilde X$ satisfies condition (vi), then this map is the identity, and we shall call $Z(\tilde X)$ for such $\tilde X$ the {\em finest twistor space} for $X$. 
\par
(Weakly) hypercomplex spaces with normal complexifications have twistor spaces:
\begin{proposition} Let $X$ be a weakly hypercomplex space such that the complexification $\tilde X$ satisfying the conditions in Definition \ref{hc-space} is normal.  Then, after making $\tilde X$ smaller, if necessary,  the ringed space  $Z(\tilde X)=\bigl((\tilde X\times \oP^1)/R,(\sO_{\tilde X\times \oP^1)})^R)$ is a complex space. Moreover, $Z(\tilde X)$ and the fibres of $\pi:Z(\tilde X)\to \oP^1$ are normal.\label{Z-exist}
\end{proposition}
\begin{proof}  Condition (i) implies that the  $R_\zeta$-equivalence classes, $\zeta\in \oP^1$,
of points in $\Reg(\tilde X)$ are of pure dimension $2n$ (possibly after making $\tilde X$ smaller). In particular, the dimension of $R$ at points in $\bigl(\Reg(\tilde X)\times\oP^1)\times \bigl(\Reg(\tilde X)\times \oP^1)$ is $4n+1$. Condition (ii) and the upper semicontinuity of dimension imply that the dimension $\dim_y L$ of any equivalence class $L$ of $R$ is $\leq 2n$ for points $y$ near $X$. Hence $\dim_{(x,y)}R\leq 4n+1$ at any $(x,y)\in R$. Now the second part of (i) and the upper semicontinuity of dimension implies that $R$ is pure-dimensional of dimension $4n+1$. Similarly, (i) and the upper semicontinuity of the dimension of fibres of the projections $p_1,p_2:R\to \tilde X\times \oP^1$ implies that the equivalence classes are pure-dimensional of dimension $2n$. Since $\tilde X$ is locally irreducible, it follows from Remmert's theorem
 \cite[Thm.\ II.1.18]{VII} that the two projections $p_1,p_2:R\to \tilde X\times \oP^1$ are open. The openness of $p_1,p_2$, the pure-dimensionality of the fibres, and the local irreducibility of $\tilde X$ mean that $R$ is a {\em nowhere degenerate} equivalence relation \cite[Def.\ IV.6.2]{VII}. Since $\tilde X$ is normal, $R$ is a {\em normal} equivalence relation (\cite[Def.\ IV.6.3]{VII} and the remarks thereafter). It follows from \cite[Thm.\ IV.7.1.]{VII}  that $Z(\tilde X)$ is a complex space, and from the remarks in the second paragraph on p.\ 203 in \cite{VII} that $Z(\tilde X)$ is normal. Finally, the normality of the fibres of $\pi$ follows from the same arguments applied to every $R_\zeta$.
\end{proof}
\begin{corollary} Let $X$ be a hypercomplex space which admits a normal complexification $\tilde X$ satisfying the conditions in Definition \ref{hc-space}. Then $\codim_\R\Sing(X)\geq 4$. 
\end{corollary}
\begin{proof} Let $Z=Z(\tilde X)$ be the corresponding twistor space. According to Proposition \ref{consequences}, the fibres of the map $X\to Z_\zeta$, where $Z_\zeta=\pi^{-1}(\zeta)$, are discrete for every $\zeta\in \oP^1$. Since $Z_\zeta$ is normal, $\codim_\R \Sing(Z_\zeta^\R)\geq 4$ and the claim follows. 
\end{proof}
Let $Z=Z(\tilde X)$ be a twistor space for a hypercomplex space $X$.
We  have a natural real analytic map $\Phi:X\times S^2\to Z^\R$, which sends $(x,\zeta)$ to the $R_\zeta$-equivalence class of $x$. We observe:
\begin{proposition} The map $\Phi:X\times S^2\to Z^\R$ has the following properties: 
\begin{enumerate}[label=(\roman*)] 
\item it is surjective with discrete fibres and unbranched  away from $\Sing(X)\times S^2$;
\item it preserves the fibres and intertwines the antipodal map on $X\times S^2$ with $\sigma$ on $Z^\R$;
\item for any $x\in X$, the map $\phi_x:S^2\to Z^\R$, given by $t\mapsto \Phi(x,t)$ is holomorphic as a map $\oP^1\to Z$;
\item the map $x\mapsto \phi_x$ from $X$ to the Douady space $\Gamma(Z)^\sigma$ of $\sigma$-invariant holomorphic sections of $\pi$ is analytic and injective. \qed
\end{enumerate}
\label{Phi}
\end{proposition}
In addition:
\begin{proposition}  Let $Z$ be a twistor space for a hypercomplex space $X$. The map $\Phi$ yields an isomorphism of $X$ with a real analytic subspace of $\Gamma(Z)^\sigma$.\label{finitePhi} \end{proposition}
\begin{proof} 
Let $\phi:X\to  \Gamma(Z)^\sigma$ be the injective  map defined by $\phi(x)(\zeta)=\Phi(x,\zeta)$. We can extend $\phi$ to an injective  holomorphic map $\tilde\phi:V\to \Gamma(Z)$, equivariant with respect to the antiholomorphic involutions, where $V$ is a neighbourhood of $X$ in $\tilde X$ with $V^\sigma=V$. Upper semicontinuity of the dimension of fibres of a holomorphic map implies that we may assume that $\tilde \phi$ has discrete fibres.
Any point $x\in X$ has a neighbourhood $U\subset V$ such that 
$\tilde\phi|_U$ is finite (i.e.\ proper). The  Finite Mapping Theorem \cite[Thm. I.8.2]{VII} implies that the image of $\tilde\phi|_U$ is a complex subspace of $ \Gamma(Z)$, and hence $\phi|_{U^\sigma}$ is an isomorphism of $U^\sigma$ with a real analytic subspace of $\Gamma(Z)^\sigma$. Since $\phi$ is injective, it is a global isomorphism between $X$ and a real analytic subspace of $\Gamma(Z)^\sigma$. 
\end{proof}
\begin{remark} This proposition remains valid for weakly hypercomplex spaces.\end{remark}
Conversely:
\begin{proposition} Let $Z$ be a reduced complex space, equipped with a surjective holomorphic map $\pi:Z\to \oP^1$ with locally irreducible fibres of dimension $2n$ and an analytic antiholomorphic involution $\sigma:Z\to Z$ covering the antipodal map.  Let $X$ be a pure-$4n$-dimensional real analytic subspace of $\Gamma(Z)^\sigma$ such that $\Reg(X)$ is open in $\Reg\bigl(\Gamma(Z)^\sigma\bigr)$. If,  for every $\zeta\in \oP^1$, the natural map $\Phi_\zeta:X\to Z_\zeta^\R$,  given by the intersection point of a section with  $Z_\zeta=\pi^{-1}(\zeta)$, is surjective, has discrete fibres, and is unbranched away from $\Sing(X)$, then $X$ is a hypercomplex space. Moreover, $Z$ is a twistor space for $X$.\label{Z-->X}
\end{proposition}
\begin{proof} Since $\Phi_\zeta$ is surjective and unbranched away from $\Sing(X)$, $\Sing(X)$ is thin. We consider a neighbourhood $\tilde X$ of $X$ inside the union of those irreducible components of $\Gamma(Z)$ which contain the irreducible components of $X$. $\tilde X$ has pure dimension $4n$. The relations $R_\zeta$ on $\tilde X$ are defined as follows: $(x,y)\in R_\zeta$ if and only if the corresponding sections intersect in the fibre $Z_\zeta=\pi^{-1}(\zeta)$ of $Z$. Conditions (iii) and (iv)  of Definition \ref{hc-space} are automatically satisfied. Condition (v)  follows from the upper-semicontinuity of the dimension of fibres of a holomorphic map. Since $\Phi_\zeta$ is surjective, every equivalence class of $R_\zeta$ intersects $X$, and we  
have to show that its dimension at such an intersection point is $2n$.  Since the fibres of $\tilde\Phi_{\zeta,\eta}:\tilde X\to Z_\zeta\times Z_\eta$, given by the intersection points of a section with the two fibres, are discrete (if $\zeta\neq\eta$), $\dim \tilde X=\dim Z_\zeta\times Z_\eta$. Since $\tilde X$ is pure-dimensional and $Z_\zeta\times Z_\eta$ is locally irreducible,  Remmert's theorem  \cite[Thm.\ II.1.18]{VII} implies that the map $\tilde\Phi_{\zeta,\eta}$ is open. Hence its composition $\tilde\Phi_\zeta$ with the projection onto $Z_\zeta$ is also open, and using    Remmert's theorem again, shows that the equivalence classes are pure $2n$-dimensional. The openness of $\tilde\Phi_{\zeta}:\tilde X\to Z_\zeta$ also implies that the relations $R_\zeta$ are the closures of their restrictions to $\Reg(\tilde X)$. 
It remains to show that these restrictions define a hypercomplex structure on $\Reg(X)$. Since $\sigma({R_\zeta})=R_{-1/\bar\zeta}\,$ for any $\zeta$, it is enough to show that the normal sheaf of a section $C$ corresponding to a point $x\in \Reg(X)$ splits as $\sO_{\oP^1}(1)^{\oplus 2n}$. 
Since $C$ is contained in the manifold $\Reg(Z)$ and the Douady space is smooth at $C$ by assumption, the normal bundle $N$ of $C$ has rank $2n$ and it satisfies  $h^0(C,N)=\dim_\C T_C \tilde X=4n$. Moreover, since $C$ is locally the only  section in $X$, and hence in $\Gamma(Z)^\sigma$,  passing  through a given point of a fibre, we conclude that $H^0(C,N(-2))=0$. It follows that $N\simeq \sO_{\oP^1}(1)^{\oplus 2n}$.\par
The last statement is a direct consequence of the definition of relations $R_\zeta$, $\zeta\in \oP^1$.
\end{proof}
\begin{proposition} Let $X$ be a connected hypercomplex space which admits a normal twistor space $Z$ such that the map $\Phi$ is finite.
If  $\Sing(X)$ does not disconnect $X$ locally, then  $\Phi$ is an isomorphism of real analytic spaces. \label{1-1}
\end{proposition}
\begin{proof} The assumptions imply that $\Reg(X)$ is connected and the map $\Phi:\Reg(X)\times S^2\to \Reg(Z)^\R$ is a proper surjective local diffeomorphism, hence a covering. The corresponding finite group $\Gamma$ of deck transformations acts fibrewise on $\Reg(X)\times S^2$, hence triholomorphically on $\Reg(X)$. But then the sections of $\Reg(Z)\to \oP^1$ are given by $\Reg(X)/\Gamma$, which implies that $\Gamma$ is the trivial group. Since $Z$ is normal and  $\Sing(X)$ does not disconnect $X$ locally, 
the classical results of Grauert and Remmert \cite[\S IV.3.2]{VII} imply that $X\times S^2$ is a normal complex space and $\Phi$ is a finite holomorphic mapping. Therefore $\Phi$ is open  \cite[Thm.\ II.1.18]{VII}, hence $1-1$ everywhere, hence a homeomorphism. Using again the normality of $Z$, we conclude that $\Phi$ is an isomorphism \cite[I.15.4]{VII}.
\end{proof}
Combining this with Proposition \ref{Z-exist}, we conclude:
\begin{corollary} Let $X$ be a connected hypercomplex space such that $\Sing(X)$ does not disconnect $X$ locally.
If the complexification $\tilde X$ in Definition \ref{hc-space} is normal and the map $\Phi$ is finite, then $X\times S^2$ is a twistor space for $X$ and $\Phi$ is an isomorphism.
\qed
\end{corollary}
In general, hypercomplex spaces such that the map is $\Phi$ is finite, have better properties. We therefore adopt the following definition.
\begin{definition} A hypercomplex space $X$ is called {\em proper} if the quotient map $X\times S^2\to (X\times S^2)/R^\sigma$ is finite, where $R^\sigma$ denotes the restriction of the relation $R$ to $X\times S^2$.\label{proper}
\end{definition}
\begin{remark} Locally, every hypercomplex space is proper.
\end{remark}
\begin{remark} If $X$ admits a twistor space and is proper, then the map $\Phi:X\times S^2\to Z^\R$, where $Z$ is the finest twistor space, is a real analytic covering (cf.\ \S\ref{ana-spaces}).
\end{remark}
\begin{proposition} Let $X$ be a proper hypercomplex space admitting a twistor space, and let $Z$ be the finest twistor space for $X$.  Then $X$ is isomorphic to the closure of a union of connected components of $\Reg\bigl(\Gamma(Z)^\sigma\bigr)$.
\end{proposition}
\begin{proof} Since $\Phi$ is proper, so is the map $\phi:X\to \Gamma(Z)^\sigma$, $\phi(x)(\zeta)=\Phi(x,\zeta)$. Therefore $X$ is a closed subspace of $\Gamma(Z)^\sigma$. Since the normal sheaf of every section in $\Reg(X)$ splits as $\sO_{\oP^1}(1)^{\oplus 2n}$, $\Reg(X)$ is open in $\Gamma(Z)^\sigma$. Hence $\Reg(X)$ is open and closed in $\Reg\bigl(\Gamma(Z)^\sigma\bigr)$ and the claim follows.
\end{proof}
\begin{remark} The closure of a connected component of $\Reg\bigl(\Gamma(Z)^\sigma\bigr)$ is not necessarily an irreducible component of $\Gamma(Z)^\sigma$. However the above proposition implies that $X$ is the intersection of $\Gamma(Z)^\sigma$ with the closure of a union of connected components of $\Reg(\tilde X)\subset \Reg(\Gamma(Z))$. This means  that $X$ is a union of irreducible components of $\Gamma(Z)^\sigma$. 
\end{remark}

Finally, let us define  $\C$-hypercomplex spaces and schemes (cf.\ \S\ref{chc-mflds}):
\begin{definition} A  complex space $(X,\sO_X)$ of pure dimension $4n$ is said to be {\em $\C$-hypercomplex} if it is equipped with a $\oP^1$-family  of equivalence relations $R_\zeta$, $\zeta\in \oP^1$, such that:
\begin{enumerate}[label=(\roman*)]
\item for every $x\in X$, $\bigcup_{\zeta\in \oP^1}[x]_\zeta\times\{\zeta\}$, where $[x]_\zeta$ denotes the $R_\zeta$-equivalence class of $x$, is a pure $(2n+1)$-dimensional closed complex subspace of a neighbourhood of $\{x\}\times \oP^1$ in $X\times \oP^1$;
\item every equivalence class $L$ of $R_\zeta$, $\zeta\in \oP^1$,  is the closure of $L\cap \Reg(X)$; 
\item the intersection of $R_\zeta$ with $\Reg(X)\times \Sing(X)$ is empty for every $\zeta\in \oP^1$. Moreover, the  restrictions of $R_\zeta$ to $\Reg(X)$ define there a $\C$-hypercomplex structure;
\item $\bigcap_{\zeta\in \oP^1} R_\zeta$ is the diagonal in $X\times X$;
\item for   every $\zeta\neq \eta$,
 the equivalence classes of $R_\zeta\cap R_\eta$ are discrete.
\end{enumerate}
\label{Chc-space}
\end{definition}
\begin{remark} We do not require that relations $R_\zeta$ are analytic, i.e.\ that their graphs are closed. In the smooth case, a $\C$-hypercomplex structure without a Hausdorff twistor space will not be given by analytic relations.
\end{remark}
\begin{remark}  {\em $\C$-hypercomplex schemes} are  defined as in Remark \ref{hc-scheme}.\end{remark}

\section{An example\label{H/pm1}}

We consider $\oH/\oZ_2$, where $\oZ_2$ acts via $q\mapsto -q$. As a singular manifold, it is a cone over $\oR {\rm P}^3$. Algebraically, over $\oR$, we can view it  as a semialgebraic variety described by invariants $x_ix_j$, $i,j=0,\dots,3$, i.e.\ the variety of symmetric $4\times 4$ real matrices of rank $\leq 1$ and nonnegative trace. Its Zariski closure is therefore the algebraic variety of  symmetric $4\times 4$ real matrices of rank $\leq 1$, which is the union of two copies of $\oH/\oZ_2$ intersecting at the origin. We shall now describe this variety as a hypercomplex space (in fact, even a hypercomplex variety).
\par
Let $Z$ be the quotient of the twistor space of $\oH$, i.e.\ of the total space of $\sO_{\oP^1}(1)\oplus \sO_{\oP^1}(1)$, by the fibrewise action of $\oZ_2$. It is a hypersurface in the total space of $\sO_{\oP^1}(2)^{\oplus 3}$ defined by $xy=z^2$. Let $\zeta$ be the affine coordinate on $\oP^1$ and $\tilde\zeta=1/\zeta$.  Then $Z$ is obtained by gluing two copies of $\cx\times (\cx^2)/\oZ_2$ with coordinates
$(\zeta,x,y,z)$, $(\tilde\zeta,\tilde x,\tilde y,\tilde z)$, $xy=z^2$, $\tilde x=x/\zeta^2$, $\tilde y=y/\zeta^2$, $\tilde z=z/\zeta^2$, and the 
real structure
$$ (\zeta,x,y,z)\mapsto \bigl(-1/\bar\zeta, \bar y/\bar\zeta^2,\bar x/\bar\zeta^2, -\bar z/\bar\zeta^2.\bigr)$$
If $(a(\zeta),b(\zeta))=(a-\bar b\zeta, b+\bar a\zeta)$ is a real section of $\sO_{\oP^1}(1)\oplus \sO_{\oP^1}(1)$, then \begin{equation}(x(\zeta),y(\zeta),z(\zeta))=\bigl(a(\zeta)^2,b(\zeta)^2,a(\zeta)b(\zeta)\bigr)\label{square}\end{equation} is a real section of $Z$. Let us now look at all real sections of $Z$. These are quadratic polynomials $(x(\zeta),y(\zeta),z(\zeta))$ of the form:
$$ x(\zeta)=x_0+x_1\zeta+x_2\zeta^2,\enskip y(\zeta)=\bar x_2-\bar x_1\zeta+\bar x_0\zeta^2,\enskip z(\zeta)=z_0+r\zeta -\bar z_0\zeta^2\enskip (r\in \R),$$
and satisfying $x(\zeta)y(\zeta)=z(\zeta)^2$. We obtain the following equations on the coefficients:
\begin{gather} x_0\bar x_2= z_0^2\label{1}\\ x_1\bar x_2-x_0\bar x_1=2rz_0\label{2}\\ |x_0|^2+|x_2|^2-|x_1|^2= r^2-2|z_0|^2.\label{3}
\end{gather}
The equation $x(\zeta)y(\zeta)=z(\zeta)^2$  implies that  $x(\zeta),y(\zeta)$ are either scalar multiples of $z(\zeta)$ or they both have double zeros. In the first case, if we write $x(\zeta)=\lambda z(\zeta)$ and $y(\zeta)=\mu z(\zeta)$, then the reality conditions imply $\mu=-\bar\lambda$, and the equation $xy=z^2$ yields $-|\lambda|^2=1$. There are, therefore, no real sections in the first case. 
In the second case, the sections satisfy, in addition, the equation 
\begin{equation} x_1^2-4x_0x_2=0.\label{4}\end{equation} The real analytic space $X$ described by these equations is locally irreducible. We want to show that is a hypercomplex space in the sense of Definition \ref{hc-space}. 
We have a map $\Phi:X\times S^2\to Z$ defined by sending $(x,\zeta)$ to the intersection point of the section with the fibre of $Z$ over $\zeta$ and, according to Proposition \ref{Z-->X}, we have to show that it is surjective with discrete fibres and unbranched away from $\Sing(X)$. The map $\Phi$ is clearly surjective. We shall show that $\Phi$ is finite. It is enough to show this for the restrictions of $\Phi$ to fibres over $S^2\simeq \oP^1$. Owing to the ${\rm SO}(3)$-invariance, it suffices to consider the restriction $\Phi_0$ to $X\times\{0\}$, which is given  by
$$\Phi_0(x_0,x_1,x_2,z_0,r)=(x_0,\bar x_2,z_0).$$ 
Equations \eqref{3} and \eqref{4} show that $\Phi_0$ is two-to-one and proper. We now verify that $\Phi$ is unbranched away from $\Sing(X)\times S^2$.
A direct inspection of the Jacobi matrix of the equations defining  $X$ shows that $\Sing(X)$ consists of the single (reduced) point $x_i=z_0=r=0$, $i=0,1,2$. Similarly, $\Sing(Z)$ equals the zero section in $\sO_{\oP^1}(2)^{\oplus 3}$. Suppose that a section in $X$ meets $\Sing(Z)$. Without loss of generality, we may assume that it meets it at $\zeta=0$. Then $x_0=x_2=z_0=0$  and equations \eqref{3} and \eqref{4} imply that $r=0$ and $x_1=0$. It follows that the branching locus of $\Phi$ is $\Sing(X)\times S^2$. Therefore $X$ is a hypercomplex space.
\par
We can say a bit more about the structure of $X$. Observe that equations \eqref{4} and \eqref{1} imply that $|x_1|^2=4|x_0x_2|=4|z_0|^2$, and then \eqref{3} yields
$$ \bigl(|x_0|-|x_2|\bigr)^2=r^2.
$$
Hence $|x_0|-|x_2|=\pm r$.
Once we know $x_0,x_2,z_0$ and $r$, then $x_1$ is uniquely determined by \eqref{2}  if $x_0\neq 0$ or $x_2\neq 0$, and by \eqref{4} otherwise.
We can therefore write $X=X_-\cup X_+$, where $|x_0|-|x_2|=- r$ on $X_-$ and  $|x_0|-|x_2|= r$ on $X_+$. 
Sections descended from 
$\sO_{\oP^1}(1)\oplus\sO_{\oP^1}(1)$ belong to $X_-$. On the other hand, $X_+$ consists of sections of the form \eqref{square}, where $a(\zeta)=a+\bar b\zeta$, $b(\zeta)=b-\bar a\zeta$ (which corresponds to the hypercomplex structure $-i,-j,k$ on $\oH$). Thus our hypercomplex space $X$ is the union of two copies of $\oH/\oZ_2$ intersecting in a point.
This decomposition is only topological, since the equations $|x_0|-|x_2|= \pm r$ are not even $C^1$.
\par
Finally, observe that if we write the sections $(a-\bar b\zeta, b+\bar a\zeta)$ or $(a+\bar b\zeta,b-\bar a\zeta)$ of $\sO_{\oP^1}(1)\oplus\sO_{\oP^1}(1)$ using real coordinates $a=x_0+ix_1$, $b=x_2+ix_3$, then we can recover the algebraic variety $X^\prime$ of $4\times 4$ real symmetric matrices of rank $\leq 1$. Indeed the coefficients in \eqref{square} yield the monomials $x_ix_j$, $i\neq j$, and the polynomials $x_i^2-x_j^2$, $i,j=0,\dots,3$. We therefore recover the traceless matrix $B=A-t/4$, where $A\in X^\prime$ and $t=\tr A$. It follows that $X$ is isomorphic to the subvariety  $\{(B,t); B(B+t/4)=0\}$ of $X_0\times \oR$, where $X_0$ are traceless symmetric $4 \times 4$ matrices. This variety is isomorphic to $X^\prime$.
\begin{remark} As mentioned in the introduction, the real algebraic variety consisting of  two copies of $\oH/\oZ_2$ intersecting in a point is also an example 
in Joyce's hypercomplex algebraic geometry (see Case (A) on p.\ 161 in \cite{Joyce}).  On the other hand, its deformation  $X_\lambda$ (Case (B), ibid.), which is homeomorphic to two copies $T^\ast S^2$ glued together along their zero sections, arises as a component  of $\Gamma(Z_\lambda)^\sigma$, where $Z_\lambda$ is a hypersurface in $\sO_{\oP^1}(2)^{\oplus 3}$ given by equation $xy=z^2+\lambda^2$, for  a fixed quadratic polynomial $\lambda(\zeta)$ with antipodal zeros. Every section in $\Sing(X_\lambda)\simeq S^2$ passes through the two singular singular points of $Z_\lambda$. This means that $X_\lambda$ is only a {\em weakly} hypercomplex space.
\end{remark}

\section{Applications}

\subsection{Finite quotients of hypercomplex manifolds}

The example discussed in the previous section indicates how to deal with general finite quotients of hypercomplex manifolds. If $M$ is such a (connected) manifold and $Z$ its twistor space, then $M$ has a natural complexification $M^\cx$ defined as the connected component of $\Gamma(Z)$ containing $M$. If a finite group $G$ acts triholomorphically on $M$, then it acts on $M^\C$ and this action commutes with $\sigma$. A classical theorem of Cartan (cf.\ \cite[Thm. I.13.14]{VII}) implies that $M^\C/G$ is a (normal) complex space. We set $\ol{M/G}=\bigl( M^\C/G\bigr)^\sigma$. The quotient $M/G$ is a closed subset of $\ol{M/G}$. We want to show that $\ol{M/G}$  is a (proper) hypercomplex space.
\par
The group $G$ acts 
fibrewise on $Z$, and the antiholomorphic involution $\sigma:Z\to Z$ is $G$-equivariant. The complex space $Z/G$ is also equipped with a projection $\pi:Z/G\to \oP^1$ and $\sigma$ descends to $Z/G$. In addition, $Z/G$ has normal fibres. It is clear that $\ol{M/G}$ is a closed subspace of $\Gamma(Z/G)^\sigma$ and the natural map $\Phi:\ol{M/G}\times S^2\to (Z/G)^\oR$ is a real analytic covering, unbranched away from $\Sing(\ol{M/G})$. Moreover, the normal bundle of a section of $Z/G$ corresponding to a point in $\Reg(\ol{M/G})$ splits as $\sO_{\oP^1}(1)^{\oplus 2n}$, so that $\Reg(\ol{M/G})$ is open in $\Reg\bigl(\Gamma(Z/G)^\sigma\bigr)$. Proposition \ref{Z-->X} implies that $\ol{M/G}$  is a hypercomplex space (proper, according to Definition \ref{proper}).
 \par
We can describe the connected components of $\Reg(\ol{M/G})$ as follows.  Elements of  $\Reg(\ol{M/G})$ correspond to free $\sigma$-invariant orbits in $M^\C$. Let $m\in M^\C$ be an element of such an orbit and let $\sigma(m)=g_mm$ for a $g_m\in G$. Then $m=\sigma^2(m)=\sigma(g_m m)=g_m\sigma(m)=g_m^2 m$, so that $g_m^2=e$. In addition, if $h\in G$, then $g_{hm}hm=\sigma(hm)=h\sigma(m)=hgm$, so that $g_{hm}=hgh^{-1}$ (since the orbit is free). For any $g\in G$ with $g^2=1$, let $Y_g$ denote the fixed point set of $g\sigma$ on $ (M^\C)^{\rm free}$. Observe that $Y_{g_1}\cap Y_{g_2}=\emptyset$ if $g_1\neq g_2$. Let $S$ denote the  set of connected components of $\bigcup \{Y_g; g^2=1\}$. The above considerations show that $G$ acts on $S$ and the connected components of $\Reg(\ol{M/G})$ are in 1-1 correspondence with $S/G$. In particular: 
\begin{proposition} If $G$ does not have elements of order $2$, then $\ol{M/G}\simeq M/G$ as real analytic spaces.\qed
\end{proposition}
\begin{remark} The idea, that in order to get an interesting structure on a real ``quotient", one has to first complexify the real manifold, is already present in \cite{HHL}. In addition, in both \cite{HHL} and in our construction,  the set of real points of the complex quotient can be larger than the real quotient we were aiming for.
\end{remark}

\subsection{Hyperk\"ahler and hypercomplex cones}

In \cite{BF} a notion of a conical hyperk\"ahler variety has been proposed and hyperk\"ahler deformations of such an object studied. It turns out, once again, that if we want a real analytic variety with a reasonable notion of a hypercomplex structure, then, in general, we cannot avoid taking the union of  two copies of such a conical variety, joined at the common vertex. We begin with a relatively simple special case.
\par
Let $M$ be a connected hypercomplex manifold with a free triholomorphic action of $\oR_{>0}$. If $E$ is the vector field generated by this action, then the vector fields $I_1E,I_2E,I_3E$ span\footnote{$I_iI_j+I_jI_i=-2\delta_{ij}{\rm Id}$ for $i,j=1,2,3$.} a Lie subalgebra of $\Gamma(TM)$ isomorphic to $\so(3)$. Let us assume that this Lie algebra generates a {\em free} action  of ${\rm SU}(2)$ or ${\rm SO}(3)$. We thus have a free action of $\oH^\ast$ or of $\oH^\ast/\oZ_2$ on $M$, such that each orbit is a hypercomplex submanifold. We remark that such hypercomplex manifolds (with  $\oH^\ast/\oZ_2$-action) arise via a bundle construction from quaternionic manifolds \cite{PPSwann}.
\par
For every complex structure $J$, corresponding to a unit imaginary quaternion, the complex manifold $(M,J)$ admits a free holomorphic action of $\cx^\ast$, generated by $E$ and $JE$. Moreover, the complex structure of $(M,J)$ does not depend on $J$.
Let $Z$ be the twistor space of $M$. Associating to each point $m$ in $M$ the twistor space $Z_m$ of its $\oH^\ast$-orbit shows that $Z$ is obtained by gluing two copies of $\cx\times (M,I_1)$ via the identification $(\tilde\zeta,\tilde m)=(\zeta^{-1},\zeta^{-l}.m)$, where $l=1$ (resp.\ $l=2$) if the orbits are isomorphic to $\oH^\ast$ (resp.\ to $\oH^\ast/\oZ^2$).
\par
Suppose now that the complex structure and the holomorphic $\C^\ast$-action on $(M,I_1)$ extend to a complex space $(\bar M,\sO_{\bar M})$, such that $\bar M=M\cup \{o\}$, where $o$ (the vertex of the cone) is the unique fixed point of the $\C^\ast$-action (cf.\ \cite[Prop. 2.1]{BF}). $\bar M$ is locally irreducible, in particular pure-dimensional. We can form a ``twistor space" $\bar Z$ by applying the above gluing construction to $\bar M$ instead of $M$. Every section of $Z_m$ is a section of $Z$, hence of $\bar Z$, and we conclude, as in \S\ref{H/pm1}, that in the case of a $\oH^\ast/\oZ_2$-action,  the irreducible component of $\Gamma(\bar Z)^\sigma$ containing $M$ is homeomorphic to the union of two copies of $\bar M^\oR$ joined at the common vertex $o$, and is, therefore, a proper hypercomplex space. In the case of a free $\oH^\ast$-action, $\bar M^{\oR}$ itself is a hypercomplex space.
\par
We can generalise the above construction to {\em stratified hypercomplex cones}. Let $Y$ be a reduced locally irreducible complex space with a holomorphic $\cx^\ast$-action, which is free except for a unique fixed  point $o\in Y$. Furthermore, suppose that the underlying real analytic space $Y^\oR$ is a stratified space, where each stratum $S_i$ is a hypercomplex manifold with a free $\oH^\ast/\oZ_2$-action, such that with respect to the complex structure $I_1$, $(S_i,I_1)$ is a $\cx^\ast$-invariant complex subspace of $Y$ (and the two $\C^\ast$-actions agree). We can then again form a ``twistor space" $Z$ via the above gluing construction (with $l=2$). The irreducible component $X$ of $\Gamma( Z)^\sigma$, containing $\Reg(Y^\R)$, is a hypercomplex space,  homeomorphic to the union of two copies of $Y$ joined at the common vertex $o$.
\par
We know of at least two classes of such stratified hypercomplex manifolds. One is provided by finite-dimensional hyperk\"ahler quotients of smooth manifolds \cite{DS,May}. If we start with a cone $M$ over a $3$-Sasakian manifold (e.g.\ $\oH^N$) with a tri-Hamiltonian action of a compact Lie group $G$ and perform the hyperk\"ahler quotient $M/\!\!/\!\!/G$ at the level set $0$ (the unique $\Ad G\times {\rm SO}(3)$-invariant point of $\g^\ast\otimes \R^3$), then we obtain a conical stratification of the complex space $(M/\!\!/\!\!/G,I_1)$, satisfying all conditions above, except perhaps the freeness of the $\oH^\ast/\oZ_2$-action. If this condition is also satisfied, then we obtain a hypercomplex space $X$, homeomorphic to the union of two copies $(M/\!\!/\!\!/G,I_1)^{\oR}$ joined at the common vertex. The other case (partially overlapping with the previous one) is the nilpotent cone $\sN$ of a complex semisimple Lie algebra. Owing to the work of Kronheimer \cite{Kron-nilp}, Biquard \cite{Biq}, and Kovalev \cite{Kov}, the affine variety $\sN$ satisfies all conditions imposed on the complex space $Y$. We therefore obtain an integral and proper hypercomplex scheme,   homeomorphic to the union of two copies of $\sN$ joined at the origin. In the case of $\ssl_2(\cx)$, this is of course the example of \S\ref{H/pm1}.

 \end{document}